\pdfoutput=1
\documentclass[letterpaper, 10pt, conference]{ieeeconf}
\IEEEoverridecommandlockouts
\usepackage{amsmath,amssymb,mathtools}
\usepackage{graphicx}
\usepackage{tikz}
\usetikzlibrary{calc,arrows.meta,positioning}
\usepackage[american]{circuitikz}
\usepackage{cite}
\usepackage[table]{xcolor}
\usepackage[acronym]{glossaries}
\definecolor{SkyBlue}{HTML}{56B4E9}
\definecolor{Vermillion}{HTML}{D55E00}

\tikzset{
  arch/base/.style={
    x=1cm, y=1cm, font=\fontsize{9}{10}\selectfont,
    line cap=round, line join=round
  },
  arch/wire/.style={draw=black, line width=0.55pt},
  arch/ext/.style={draw=SkyBlue, line width=0.9pt},
  arch/port/.style={
    circle, fill=white, draw=SkyBlue, line width=0.9pt,
    inner sep=1.1pt, outer sep=0pt
  },
  arch/junction/.style={
    circle, fill=Vermillion, draw=Vermillion, line width=0.9pt,
    inner sep=1.1pt, outer sep=0pt
  },
  arch/var/.style={
    circle, draw=black, fill=white, line width=0.55pt,
    minimum size=6.8mm, inner sep=0pt
  },
  arch/shared/.style={
    arch/var, draw=black, fill=Vermillion!25!white,
    text=white, line width=0.9pt, dashed
  },
  arch/factor/.style={
    rectangle, draw=black, fill=gray!25!white, line width=0.55pt,
    minimum width=6.8mm, minimum height=6.8mm, inner sep=1pt
  },
  arch/panel/.style={anchor=base, inner sep=0pt}
}

\newcommand{\ArchitecturePanelLabels}[3]{%
  \node[arch/panel] at (#1\columnwidth,-0.67) {Closed};
  \node[arch/panel] at (#2\columnwidth,-0.67) {Open};
  \node[arch/panel] at (#3\columnwidth,-0.67) {Interconnection};
}
\newcommand{\ArchitectureCanvas}{%
  \pgfresetboundingbox
  \path[use as bounding box]
    (0,-0.80) rectangle (\columnwidth,1.60);
}

\makeglossaries
\newacronym{QP}{QP}{quadratic program}
\newacronym{KKT}{KKT}{Karush--Kuhn--Tucker}
\newacronym{LTI}{LTI}{linear time-invariant}
\newacronym{VI}{VI}{variational inference}
\newacronym{KL}{KL}{Kullback--Leibler}

\newtheorem{theorem}{Theorem}
\newtheorem{lemma}[theorem]{Lemma}
\newtheorem{proposition}[theorem]{Proposition}

\newtheorem{remark}[theorem]{Remark}
\newtheorem{example}[theorem]{Example}


\newenvironment{proof}[1][Proof]{\par\noindent{\itshape #1: }}{\hfill${\blacksquare}$\par\smallskip}

\DeclareMathOperator*{\argmin}{arg\,min}
\DeclareMathOperator{\graph}{graph}
\DeclareMathOperator{\im}{image}
\DeclareMathOperator{\rank}{rank}
\DeclareMathOperator{\diag}{diag}
\DeclareMathOperator{\Ent}{h}
\DeclareMathOperator{\KL}{KL}
\newcommand{\spec}{\sigma}
\newcommand{\R}{\mathbb{R}}
\newcommand{\E}{\mathbb{E}}
\newcommand{\tp}{^{\mathsf{T}}}
\newcommand{\1}{\mathbf{1}}
\newcommand{\B}{\mathcal{B}}
\newcommand{\W}{\mathbb{W}}

\title{\LARGE \bf Open Quadratic Optimization Programs}

\author{Alberto Padoan%
\thanks{The author is with the Department of Electrical and Computer Engineering, University of British Columbia, Vancouver, BC, Canada. {\tt\small alberto.padoan@ubc.ca}. This work was supported by NSERC (RGPIN-2025-06895, DGECR-2025-00382).}%
}

\begin{document}
\maketitle
\thispagestyle{empty}
\pagestyle{empty}

\begin{abstract}
The paper studies open optimization programs, that is, parametric optimization programs viewed as open systems, and their interconnection. Each open optimization program is assigned a behavior and interconnection is defined through variable sharing. For a class of open \glspl{QP}, we characterize closure and well-posedness under interconnection and, for modules interconnected by adding their objectives, analyze the convergence of a sequential module-by-module solution method, deriving error bounds for the resulting approximation. We further illustrate that the variables through which modules are interconnected indicate an appropriate behavioral description, and each description requires dedicated interconnection rules. This perspective suggests a route toward compositional analysis and design of optimization architectures.
\end{abstract}

\section{Introduction}\label{sec:intro}

Optimization problems are rarely defined in isolation~\cite{DorflerHeBelgioioso2024}. In control architectures, an estimator supplies a state estimate to a planner, which in turn supplies a reference to a controller~\cite{matni2024controlarchitecture}. In physical systems, components exchange variables under conservation laws, with equilibria often characterized by minimizing an energy or dissipation functional~\cite{DoyleSnell1984,vdSJeltsema2014}. In probabilistic inference, modules exchange beliefs over shared latent variables while minimizing free energy~\cite{WainwrightJordan2008TR,Blei2017}. Once optimization programs are interconnected, however, solving each one in isolation is no longer a central question: one must understand instead the behavior generated by their interconnection. Crucially, such behavior depends not only on which programs are interconnected, but on \emph{how} they are interconnected and \emph{which variables} they exchange.

Systems theory provides a natural language to study such questions~\cite{matni2024controlarchitecture}.
Recent work has developed this perspective for optimization \textit{algorithms}~\cite{DorflerHeBelgioioso2024,eising_open_optimization}, viewing them as systems interacting with other algorithms and their environment.
This paper takes a complementary equilibrium perspective, modeling
\emph{optimization programs as open systems} with parameters as inputs,
optimizers as outputs, and corresponding input--output relations as the objects
of interconnection.
For convex \glspl{QP}, we ask when these interconnections are well posed and closed, how their solutions can be assembled from module-level computations, and which interface variables are sufficient to compose relevant quantities (\textit{e.g.},  optimizers, conjugate variables, and optimal values).

\subsubsection*{Contributions}
We model open optimization programs by their behavior, in the sense of behavioral systems theory~\cite{Willems2007,PolWil1998}, and define interconnection through variable sharing (Section~\ref{sec:oqp}). For a class of open \glspl{QP} with nonsingular \gls{KKT} matrices, we show that input--output behaviors are affine, characterize closure and well-posedness under interconnection by consistency and rank conditions that reduce to an eigenvalue test for single-loop architectures, and, for modules interconnected by adding their objectives, identify sequential module evaluation with a block Gauss--Seidel iteration~\cite{Varga2000,GolubVanLoan2013}, deriving convergence conditions and error bounds. For interfaces pairing conjugate variables, as in resistive electrical networks, we introduce the primal--dual behavior as the graph of the value-function gradient. We note that its Jacobian is a Schur complement of the \gls{KKT} matrix, that its graph has a Lagrangian direction space, so that port interconnection composes these behaviors by partial infimal convolution of value functions and Schur complementation of their Hessians (Section~\ref{sec:pd}). Finally, variational inference motivates augmenting the primal--dual behavior with the optimal value, since composing free energies requires the additive constant lost under differentiation. We show that the resulting behavior is a Legendrian submanifold, derive the addition law for conjugate variables and optimal values, and identify the correction that prevents the entropy of a shared belief from being counted twice (Section~\ref{sec:prob}).

\subsubsection*{Related work}
Interconnection as variable sharing is a central construction of behavioral systems theory~\cite{Willems2007,PolWil1998}. Compositional modeling of affine relations, signal-flow diagrams, and open linear systems appear in~\cite{BonchiEtAl2019,FongRapisardaSobocinski2016}. Classical sensitivity analysis and multiparametric programming describe how optimizers and values depend on parameters~\cite{Danskin1966,bonnans2000perturbation,BemporadMorari2002}. Duality and infimal convolution provide composition rules for convex functions~\cite{rockafellar1970convex,bauschke2017convex,SteinEtAl2026}. Numerical analysis of \gls{KKT} systems, block relaxation, and Schur complements is classical~\cite{BenziGolubLiesen2005,Varga2000,GolubVanLoan2013,zhang2005schur}, with Kron reduction providing the network counterpart of Schur complementation~\cite{DorflerBullo2013,FB-LNS}. Electrical realizations of \glspl{QP} can be traced to~\cite{dennis1959}. Port-Hamiltonian modeling emphasizes interconnection through conjugate port variables~\cite{vdSMaschke2013,vdSJeltsema2014}. Category theory describes passive networks by Lagrangian relations~\cite{BaezFong2018,ComfortKissinger2022,BoothEtAl2026}. Legendrian manifolds and their composition are standard in contact geometry~\cite{arnold1989,Limouzineau2016}. Architectural perspectives appear in network decomposition, layered control, and categorical formulations of optimal control~\cite{LowLapsley1999,SangiovanniVincentelli2007,ChiangLowCalderbankDoyle2007,matni2024controlarchitecture,HanksEtAl2024,SheEtAl2023}.

\subsubsection*{Paper organization}
Section~\ref{sec:apps} presents three motivating applications. Section~\ref{sec:oqp} develops our system-theoretic viewpoint, studying closure, well-posedness, and properties of a class of open \glspl{QP}. Sections~\ref{sec:pd} and~\ref{sec:prob} extends the framework to primal--dual and value descriptions, together with their interconnection rules. Proofs are collected in the Appendix.

\section{Motivating applications}\label{sec:apps}

\subsection{Control architectures}\label{ssec:control_architecture}

Modern control systems are typically organized as layered architectures, with estimation, planning, and control performed by separate modules that exchange information through well-defined interfaces~\cite{matni2024controlarchitecture,DorflerHeBelgioioso2024}.
We consider a simple architecture where each module solves a \gls{QP} and exchanged signals serve as interface variables, as illustrated in Fig.~\ref{fig:control_stack}.

\begin{figure}[h!]
\centering
\resizebox{\columnwidth}{!}{%
\begin{tikzpicture}
\fill [ fill=orange!25!white, thick, rounded corners]
      (1.5,7.2) rectangle (9.4,4.8) {};
\draw [fill=gray!25!white, draw=black, thick, rounded corners]
      (-0.5,6.5) rectangle (1,5.5)
      node[midway] {\textsf{system}};
\draw [ultra thick,-latex](2.3,6) -- (1,6);
\node at (2.0,6.28) {$u$};
\draw [fill=gray!25!white, draw=black, thick, rounded corners]
      (2.3,6.5) rectangle (3.8,5.5)
      node[midway] {\textsf{control}};
\draw [ultra thick,-latex](4.8,6) -- (3.8,6);
\node at (4.4,6.28) {$x_{\mathrm{ref}}$};
\draw [fill=gray!25!white, draw=black, thick, rounded corners]
      (4.8,6.5) rectangle (6.3,5.5)
      node[midway] {\textsf{planning}};
\draw [ultra thick,-latex](7.3,6) -- (6.3,6);
\node at (6.9,6.28) {$\hat x_1$};
\draw [fill=gray!25!white, draw=black, thick, rounded corners]
      (7.3,6.5) rectangle (8.8,5.5)
      node[midway] {$\begin{array}{cc}
       {\scalebox{.9}{$\textsf{state}$}}  \\[-2.5pt]
       {\scalebox{.9}{$\textsf{estimation}$}}
      \end{array}$};
\draw [black, ultra thick,-latex, rounded corners]
      (-0.5,6) -- (-1,6) -- (-1,4.5) -- (9.5,4.5) -- (9.5,6) -- (8.8,6);
\node at (9.25,6.28) {$\tilde x$};
\draw [black, ultra thick,-latex, rounded corners](6.8,6) -- (6.8,5) -- (3.05,5) -- (3.05,5.5);
\draw [black, ultra thick,-latex, rounded corners](1.8,6) -- (1.8,7) -- (8.05,7) -- (8.05,6.5);
\end{tikzpicture}}
\caption{Example of a modular control architecture (shaded) consisting of estimation, planning, and control modules interacting with a given system.}
\label{fig:control_stack}
\end{figure}
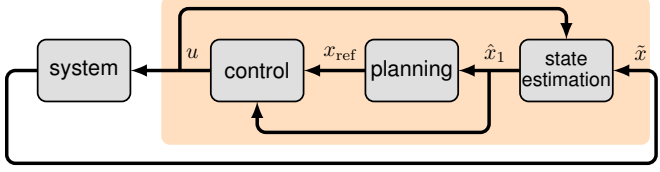

Consider a discrete-time \gls{LTI} system described by the equations
\begin{equation}\label{eq:system_linear}
x_{t+1} = F x_{t} + G u_{t},
\end{equation}
where ${x_t\in\R^n}$ and ${u_t\in\R^m}$ denote the state and input of the system over the time horizon ${\mathbb{T}=\{1,\ldots,T\}}$, respectively.

We now specify the three optimization modules.
The estimation module receives a noisy state sequence ${\tilde x\in\R^{nT}}$ and an input sequence ${u\in\R^{mT}}$, producing an initial state estimate ${\hat x_1\in\R^n}$. Assuming i.i.d. zero-mean Gaussian measurement noise with covariance ${\Sigma\succ0}$, the maximum-likelihood trajectory estimate ${\hat x\in\R^{nT}}$ is obtained by solving
\begin{equation}\label{eq:estimation_cls}
\min_{x\in\mathcal{X}(u)}\ \|x-\tilde x\|^2_{I \otimes \Sigma^{-1}},
\end{equation}
where ${\|z\|_M^2=z\tp Mz}$, ${\otimes}$ is the Kronecker product, and
\[
\mathcal{X}(u)=
\{\,x\in\R^{nT}\mid
x_{t+1}=Fx_t+Gu_t,\, t=1,\ldots,T-1\,\}.
\]
For each ${u\in\R^{mT}}$, the set ${\mathcal{X}(u)}$ is affine. Since ${\Sigma\succ 0}$, problem~\eqref{eq:estimation_cls} is a strictly convex \gls{QP} with a unique minimizer.

The planning module receives the initial state estimate
${\hat x_1}$ and an exogenous target state
${x_{\mathrm{goal}}\in\R^n}$, producing a dynamically feasible
reference trajectory ${x_{\mathrm{ref}}\in\R^{nT}}$.
A natural formulation is
\begin{equation}\label{eq:planning_cls}
\begin{alignedat}{2}
&\min_{(x,u)\in\R^{nT}\times\R^{mT}}\quad
&&\|x_T-x_{\mathrm{goal}}\|_P^2+\|u\|_S^2\\
&~~~~~~~\,\textup{s.t.}\quad
&&x\in\mathcal{X}(u),\qquad x_1=\hat x_1,
\end{alignedat}
\end{equation}
where ${P\succeq0}$ and ${S\succ0}$ penalize terminal error and control
effort, respectively. Problem~\eqref{eq:planning_cls} is again a \gls{QP}
parameterized by ${\hat x_1}$ and ${x_{\mathrm{goal}}}$. Since the
dynamics uniquely determine ${x}$ from ${(u,\hat x_1)}$ and ${S\succ0}$,
the minimizer is unique.

The control module receives the reference trajectory
${x_{\mathrm{ref}}}$ and the initial state estimate
${\hat x_1}$, producing an input sequence ${u\in\R^{mT}}$ by solving the problem
\begin{equation}\label{eq:control_cls}
\begin{alignedat}{2}
&\min_{(x,u)\in\R^{nT}\times\R^{mT}}\quad
&&\|x-x_{\mathrm{ref}}\|_Q^2+\|u\|_{R}^2\\
&~~~~~~~\,\textup{s.t.}\quad
&&x\in\mathcal{X}(u),\qquad x_1=\hat x_1,
\end{alignedat}
\end{equation}
where ${Q\succeq0}$ and ${R\succ0}$ are weighting matrices.
Problem~\eqref{eq:control_cls} is a \gls{QP} parameterized by ${x_{\mathrm{ref}}}$ and ${\hat x_1}$, with a unique minimizer, since~\eqref{eq:system_linear} determines ${x}$ from ${(u,\hat x_1)}$ and ${R\succ0}$.

For fixed exogenous data, the system and the measurement process together map ${u}$ to the noisy state sequence ${\tilde x}$, closing the loop via the estimation module. This raises a first question: when is the architecture well posed, and can its behavior be derived from those of the individual modules?

\subsection{Resistive electrical networks}\label{ssec:resistive_networks}

Physical systems often admit variational descriptions based on energy
and dissipation principles subject to conservation
laws~\cite{vdSJeltsema2014}. Following~\cite{FB-LNS}, we use resistive electrical networks as a
canonical example, whose current flows minimize power dissipation
according to \emph{Thomson's principle}~\cite{DoyleSnell1984}.

Consider a resistive electrical network modeled by a connected, undirected, weighted graph ${\mathbb{G}=(\mathbb{V},\mathbb{E})}$, with node set ${\mathbb{V}:=\{1,\ldots,n\}}$, edge set ${\mathbb{E}:=\{1,\ldots,m\}}$, and edge resistances ${r_{ij}>0}$ for each ${\{i,j\}\in\mathbb{E}}$.
Let ${c\in\R^n}$ denote the vector of external current injections and assume ${\1\tp c=0}$, where ${\1\in\R^n}$ is the vector of ones.
This condition balances total injected and extracted current and is necessary for solvability of Kirchhoff's current law~\cite{FB-LNS}. 
Fix an arbitrary orientation of the edges and let ${f\in\R^m}$ collect the directed edge flows ${f_{i\to j}}$, with ${f_{j\to i}=-f_{i\to j}}$ for ${\{i,j\}\in\mathbb{E}}$ and ${f_{i\to j}=0}$ otherwise.
Let ${B\in\R^{n\times m}}$ be the incidence matrix of the graph ${\mathbb{G}}$ and let ${R:=\diag(r_{ij})\in\R^{m\times m}}$ be the diagonal, positive definite matrix of edge resistances. With this setup, Kirchhoff's current law is
\begin{equation}\label{eq:kcl}
Bf=c.
\end{equation}
Thomson's principle states that the \emph{physical flow}, that is, the current distribution realized by the network at equilibrium, minimizes dissipated power among all flows satisfying~\eqref{eq:kcl}, and is thus the solution of the \gls{QP}
\begin{equation}\label{eq:thomson_cls}
\min_{f\in\R^{m}}\ \frac12 f\tp R f
\qquad
\textup{s.t.}\quad Bf=c.
\end{equation}
The corresponding \gls{KKT} conditions are
\begin{equation}\label{eq:thomson_kkt}
Rf+B\tp v=0,
\qquad
Bf=c,
\end{equation}
where ${v\in\R^n}$ is a Lagrange multiplier.
Interpreting ${v}$ as the vector of nodal potentials, the first equation in~\eqref{eq:thomson_kkt} is Ohm's law on each edge, while eliminating ${f}$ yields the weighted Laplacian equation ${Lv=-c}$, where ${L:=BR^{-1}B\tp}$ is the weighted graph Laplacian of ${\mathbb{G}}$~\cite{FB-LNS}.
Since the graph ${\mathbb{G}}$ is connected, ${\rank L=n-1}$ and ${\ker L=\mathrm{span}\{\1\}}$, so ${v}$ is unique up to an additive constant, while the minimizer ${f^\star}$ of~\eqref{eq:thomson_cls} is unique because ${R\succ0}$.

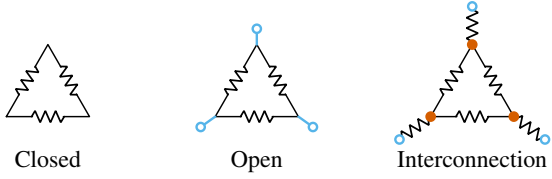
\begin{figure}[t]
\centering
\begin{circuitikz}[arch/base]
  \ctikzset{resistors/scale=0.6, bipoles/length=0.85cm,
            bipoles/thickness=1.2}

  \begin{scope}[xshift=0.15\columnwidth, scale=0.55]
    \draw[arch/wire]
      (-1,0) to[R] (1,0) to[R] (0,1.732) to[R] (-1,0);
  \end{scope}

  \begin{scope}[xshift=0.47\columnwidth, scale=0.55]
    \draw[arch/wire]
      (-1,0) to[R] (1,0) to[R] (0,1.732) to[R] (-1,0);
    \foreach \x/\y/\dx/\dy in {
      -1/0/-0.35/-0.25,
       1/0/ 0.35/-0.25,
       0/1.732/0/0.38}{
      \draw[arch/ext] (\x,\y) -- (\x+\dx,\y+\dy);
      \node[arch/port] at (\x+\dx,\y+\dy) {};
    }
  \end{scope}

  \begin{scope}[xshift=0.80\columnwidth, scale=0.55]
    \draw[arch/wire]
      (-1,0) to[R] (1,0) to[R] (0,1.732) to[R] (-1,0);
    \draw[arch/wire] (-1,0)    to[R] (-1.75,-0.55);
    \draw[arch/wire] ( 1,0)    to[R] ( 1.75,-0.55);
    \draw[arch/wire] (0,1.732) to[R] (0,2.65);
    \foreach \x/\y in {-1.75/-0.55,1.75/-0.55,0/2.65}
      \node[arch/port] at (\x,\y) {};
    \foreach \x/\y in {-1/0,1/0,0/1.732}
      \node[arch/junction] at (\x,\y) {};
  \end{scope}

  \ArchitecturePanelLabels{0.15}{0.47}{0.80}
  \ArchitectureCanvas
\end{circuitikz}
\caption{Closed (left), open (center), and interconnected (right) resistive networks.
Boundary terminals ({\color[HTML]{56B4E9}${\circ}$}) exchange currents and potentials
with the environment. Shared nodes ({\color[HTML]{D55E00}${\bullet}$}) mark 
interconnections.}
\label{fig:resistive_networks}
\end{figure}

Program~\eqref{eq:thomson_cls} models a \emph{closed} resistive network. 
A circuit is, however, naturally regarded as an \emph{open} system interacting with an environment (Fig.~\ref{fig:resistive_networks}).
We therefore partition the node set ${\mathbb{V}}$ into \emph{internal} and \emph{external} nodes: the former have prescribed current injections, while the latter serve as boundary terminals at which currents and potentials are exchanged with the environment.
Accordingly, we write
\begin{equation}\label{eq:resistive_partition}
c=\begin{bmatrix}c_{\mathrm{i}}\\c_{\mathrm{e}}\end{bmatrix},
\
v=\begin{bmatrix}v_{\mathrm{i}}\\v_{\mathrm{e}}\end{bmatrix},
\
B=\begin{bmatrix}B_{\mathrm{i}}\\B_{\mathrm{e}}\end{bmatrix},
\
L=\begin{bmatrix}L_{\mathrm{ii}}&L_{\mathrm{ie}}\\L_{\mathrm{ei}}&L_{\mathrm{ee}}\end{bmatrix},
\end{equation}
where the subscripts ${\mathrm{i}}$ and ${\mathrm{e}}$ indicate the internal and external nodes, respectively.
Solving ${Lv=-c}$ for ${v_{\mathrm{i}}}$ yields
\begin{equation}\label{eq:boundary_relation}
c_{\mathrm{e}}
=-\big(L_{\mathrm{ee}}-L_{\mathrm{ei}}L_{\mathrm{ii}}^{\dagger}L_{\mathrm{ie}}\big)v_{\mathrm{e}}
+L_{\mathrm{ei}}L_{\mathrm{ii}}^{\dagger}c_{\mathrm{i}},
\end{equation}
where ${{M}^\dagger}$ is the Moore--Penrose pseudoinverse of ${M\in\R^{p\times q}}$.
Eliminating the internal degrees of freedom yields the external representation~\eqref{eq:boundary_relation}
of the network that depends on its internal structure only through the Schur
complement of ${L}$. In circuit theory, this is known as Kron
reduction~\cite{DorflerBullo2013}.

Interconnecting two open resistive networks along their terminals enforces Kirchhoff's current and voltage laws,
\begin{equation}\label{eq:port_interconnection}
c_{1,\mathrm{e}}+c_{2,\mathrm{e}}=0,
\qquad
v_{1,\mathrm{e}}=v_{2,\mathrm{e}}.
\end{equation}
In contrast with Section~\ref{ssec:control_architecture}, the interconnection~\eqref{eq:port_interconnection} constrains \emph{both} the primal variables ${c}$ and the dual variables ${v}$ of the two \glspl{QP}. It relates a conjugate pair of variables of one module to the conjugate pair of the other, and is therefore defined by a \emph{pair} of equations rather than by a single one.
This raises a second question: which interconnections are described by a single equation between variables of adjacent modules, which require a pair of equations between conjugate variables, and what distinguishes the two?

\subsection{Variational inference}\label{ssec:vi}

Many inference tasks in machine learning may be phrased as optimization programs~\cite{WainwrightJordan2008TR}.
A canonical example is \gls{VI}, in which one approximates a posterior distribution by minimizing a free-energy functional~\cite{WainwrightJordan2008TR,Blei2017}.

Let ${p(x,z)}$ be the joint density of observed data ${x}$ and latent
variables ${z}$. Assume the evidence ${p(x):=\int p(x,z)\,dz}$ is positive and bounded everywhere, and define the posterior ${p(z\mid x):=p(x,z)/p(x)}$.
For densities ${q,p}$, the \gls{KL} divergence is
${\KL(q\,\|\,p):=\E_q[\log q-\log p]}$,
where ${\E_q}$ denotes expectation under ${q}$~\cite{CoverThomas2006}.
Let ${\mathcal{Q}}$ be a family of densities over ${z}$.
For ${q\in\mathcal{Q}}$ with finite \gls{KL} divergence from the posterior, direct
substitution gives~\cite{WainwrightJordan2008TR,Blei2017}
\begin{equation*}
\KL(q(z)\,\|\,p(z\mid x))
=\E_q\big[\log q(z)-\log p(x,z)\big]+\log p(x).
\end{equation*}
Consequently, minimizing the divergence over ${\mathcal{Q}}$ is equivalent to minimizing the \emph{variational free energy}
\begin{equation}\label{eq:vi_free_energy}
\mathcal{F}(q):=\E_q\big[\log q(z)-\log p(x,z)\big].
\end{equation}
A \gls{VI} problem is therefore the optimization program
\begin{equation}\label{eq:vi_closed}
\min_{q\in\mathcal{Q}}\ \mathcal{F}(q).
\end{equation}
Parameterizing ${\mathcal{Q}}$ by a vector ${\theta\in\Theta}$, with ${\Theta\subseteq\R^{\ell}}$, and writing ${q=q_\theta}$ induces the objective ${\Phi(\theta):=\mathcal{F}(q_\theta)}$. For a Gaussian family with fixed covariance and mean ${\theta}$,
and a log-quadratic model, ${\Phi}$ is a strictly convex quadratic
function of ${\theta}$. If ${\Theta}$ is affine,~\eqref{eq:vi_closed} is a \gls{QP}.

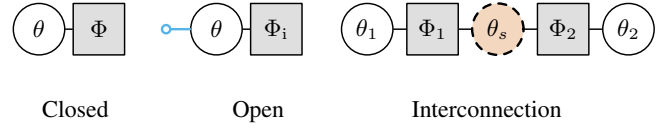
\begin{figure}[t]
\centering
\begin{tikzpicture}[arch/base]

  \begin{scope}[scale = 0.9, xshift=0.10\columnwidth, yshift=0.55cm]
    \node[arch/var]    (closed-t) at (-0.44,0) {$\theta$};
    \node[arch/factor] (closed-f) at ( 0.44,0) {$\Phi$};
    \draw[arch/wire] (closed-t) -- (closed-f);
  \end{scope}

  \begin{scope}[scale = 0.9, xshift=0.38\columnwidth, yshift=0.55cm]
    \node[arch/var]    (open-t) at (-0.25,0) {$\theta$};
    \node[arch/factor] (open-f) at ( 0.63,0) {$\Phi_{\mathrm{i}}$};
    \draw[arch/wire] (open-t) -- (open-f);
    \draw[arch/ext] (open-t) -- (-0.99,0);
    \node[arch/port] at (-0.99,0) {};
  \end{scope}

  \begin{scope}[scale=.9, xshift=0.83\columnwidth, yshift=0.55cm]
    \node[arch/var]    (t1) at (-1.92,0) {$\theta_1$};
    \node[arch/factor] (f1) at (-0.96,0) {$\Phi_1$};
    \node[arch/shared] (ts) at ( 0,0)    {\textcolor{black}{$\theta_s$}};
    \node[arch/factor] (f2) at ( 0.96,0) {$\Phi_2$};
    \node[arch/var]    (t2) at ( 1.92,0) {$\theta_2$};
    \foreach \a/\b in {t1/f1,f1/ts,ts/f2,f2/t2}
      \draw[arch/wire] (\a) -- (\b);
  \end{scope}

  \ArchitecturePanelLabels{0.10}{0.38}{0.73}
  \ArchitectureCanvas
\end{tikzpicture}
\caption{Closed (left), open (center), and interconnected (right) inference modules. 
Boundary terminals ({\color[HTML]{56B4E9}${\circ}$}) exchange conjugate variables with the environment. Shared nodes ({\color[HTML]{D55E00} dashed}) mark interconnections.}
\label{fig:factor}
\end{figure}

Program~\eqref{eq:vi_closed} models a \emph{closed} \gls{VI}  inference problem. 
In a modular inference pipeline, this is rarely the case: the objective splits as ${\Phi=\Phi_{\mathrm{i}}+\Phi_{\mathrm{e}}}$, with ${\Phi_{\mathrm{i}}}$ internal to the module and ${\Phi_{\mathrm{e}}}$ contributed by the environment, and the conjugate variables ${y:=\nabla\Phi_{\mathrm{i}}(\theta)}$ and ${u:=\nabla\Phi_{\mathrm{e}}(\theta)}$ are exchanged through a port, as illlustrated in Fig.~\ref{fig:factor} by a factor graph~\cite{KFL2001} where circles are beliefs and squares are potentials.
Two modules with potentials ${\Phi_1(\theta_s,\theta_1)}$ and ${\Phi_2(\theta_s,\theta_2)}$ sharing the belief ${\theta_s}$ are interconnected by the program
\begin{equation}\label{eq:composed_problem}
\min_{\theta_s,\theta_1,\theta_2}\ \Phi_1(\theta_s,\theta_1)+\Phi_2(\theta_s,\theta_2),
\end{equation}
whose stationarity condition in ${\theta_s}$ is the balance law ${y_1+y_2=0}$, with ${y_i:=\nabla_{\theta_s}\Phi_i}$.
The \emph{potentials} in~\eqref{eq:composed_problem} are chosen so that each contribution to the free energy is counted once, whereas adding the free energies of the two modules, each of which accounts for the shared belief, would count their common terms twice.
In contrast with Sections~\ref{ssec:control_architecture} and~\ref{ssec:resistive_networks}, the interconnection thus also composes \emph{values}, the free energies, in addition to variables and their conjugates.
This raises a third question: how do values compose under interconnection, and what must be transmitted so that they compose without double counting?

\section{Open optimization programs}\label{sec:oqp}

Consider the parametric optimization program
\begin{equation}\label{eq:open}
\begin{array}{cl}
\displaystyle\min_{x\in\R^{n}} & f(x,u)\\[3pt]
\textup{s.t.} & h(x,u)=0,\quad g(x,u)\le0,
\end{array}
\end{equation}
with objective ${f:\R^n\times\R^m\to\R}$ and constraints defined by ${h:\R^n\times\R^m\to\R^{p}}$ and ${g:\R^n\times\R^m\to\R^{s}}$.
We regard the parameter ${u\in\R^m}$ as an input supplied by an environment, and we call~\eqref{eq:open} an \emph{open optimization program}.
For each ${u\in\R^m}$, the \emph{feasible set} of~\eqref{eq:open} is the set of all ${x\in\R^n}$ satisfying the constraints of~\eqref{eq:open},
\begin{equation}\label{eq:feasible}
\mathcal{X}(u):=\{\,x\in\R^n\mid h(x,u)=0,\ g(x,u)\le0\,\},
\end{equation}
the \emph{value function} ${V:\R^m\to\R\cup\{\pm\infty\}}$ is
\begin{equation}\label{eq:value}
V(u):=\inf_{x\in\mathcal{X}(u)}f(x,u),
\end{equation}
and the \emph{solution map} ${X^\star}$ is the relation
\begin{equation}\label{eq:solmap}
u\mapsto X^\star(u):=\argmin_{x\in\mathcal{X}(u)}f(x,u).
\end{equation}
For ${f}$, ${g}$, and ${h}$ differentiable, ${f}$ and ${g}$ convex in the first argument for each ${u}$, and ${h}$ affine in the first argument,~\eqref{eq:open} is convex and, under a constraint qualification (e.g.,  Slater's condition), the \gls{KKT} conditions are necessary and sufficient for optimality~\cite{BoydVandenberghe2004,NocedalWright2006}.
Then ${X^\star(u)}$ is the projection onto the first components of the set of \gls{KKT} solutions of~\eqref{eq:open}.
Let ${\mathcal{L}(x,\lambda,\mu,u):=f(x,u)+\lambda\tp h(x,u)+\mu\tp g(x,u)}$ be the Lagrangian of~\eqref{eq:open}, with multipliers ${\lambda\in\R^{p}}$ and ${\mu\in\R^{s}_{+}}$. We call ${\nabla_u\mathcal{L}(x,\lambda,\mu,u)}$, evaluated at a primal--dual solution ${(x,\lambda,\mu)}$, the \emph{conjugate variable} of the input. 
If one assumes further that ${f}$ and ${g}$ are jointly convex in both arguments and that ${h}$ is affine in both arguments, then ${V}$ is convex, being the partial minimization of a jointly convex function.
Moreover, the conjugate variable is a subgradient of ${V}$ at ${u}$~\cite{Danskin1966,bonnans2000perturbation}.

In this paper, we focus on open optimization programs with quadratic objectives and equality constraints. Inequality constraints may be accommodated through the active set at a solution, under which the solution depends piecewise affinely on the parameter, as in multiparametric programming~\cite{BemporadMorari2002}.

\subsubsection{Open quadratic optimization programs}\label{ssec:oqp}
An \emph{open quadratic optimization program} is an open optimization program of the form
\begin{equation}\label{eq:oqp}
\begin{array}{cl}
\displaystyle\min_{x\in\R^{n}} & \dfrac12 x\tp Px+(q+Nu)\tp x\\[8pt]
\textup{s.t.} & Ax+Mu=b,
\end{array}
\end{equation}
with ${P\in\R^{n\times n}}$ symmetric, ${q\in\R^n}$, ${N\in\R^{n\times m}}$, ${A\in\R^{p\times n}}$, ${M\in\R^{p\times m}}$, and ${b\in\R^{p}}$, and with input ${u\in\R^m}$.
We assume throughout the paper that ${P}$ is positive semidefinite.
Then~\eqref{eq:oqp} is a convex program, whose \gls{KKT} conditions characterize optimality~\cite{BoydVandenberghe2004} and read
\begin{equation}\label{eq:kkt}
K\begin{bmatrix}x\\\lambda\end{bmatrix}
=\begin{bmatrix}-(q+Nu)\\b-Mu\end{bmatrix},
\qquad
K:=\begin{bmatrix}P&A\tp\\A&0\end{bmatrix},
\end{equation}
where ${\lambda\in\R^{p}}$ is a Lagrange multiplier.
The matrix ${K}$ in~\eqref{eq:kkt} is the \gls{KKT} matrix of~\eqref{eq:oqp}, and systems of this form are well studied~\cite{BenziGolubLiesen2005}.
We say that the input \emph{affects}~\eqref{eq:oqp} through the constraints if ${N=0}$, and through the objective if ${M=0}$.

\subsection{The behavior of an open quadratic optimization program}\label{ssec:behavior}

 In behavioral systems theory~\cite{Willems2007,PolWil1998}, a system is a triple ${\Sigma=(\mathbb{T},\W,\B)}$, where ${\mathbb{T}}$ is the time set, ${\W}$ is the signal space, and the behavior ${\B\subseteq\W^{\mathbb{T}}}$ is the set of all trajectories of the system.
The open programs studied in this paper are static, so ${\mathbb{T}}$ is a singleton and a system is modeled by a pair ${(\W,\B)}$, with ${\B\subseteq\W}$. We associate with~\eqref{eq:oqp} the signal space ${\W:=\R^m\times\R^n}$, whose elements are the pairs ${w=(u,x)}$ of input and decision variable, and the behavior ${\B:=\graph X^\star}$, the graph of the solution map~\eqref{eq:solmap}.

\begin{lemma}[Affine behavior]\label{lem:affine}
Consider the open quadratic optimization program~\eqref{eq:oqp}, with \gls{KKT} matrix ${K}$ as in~\eqref{eq:kkt} and behavior ${\B=\graph X^\star}$.
Assume ${P\succeq0}$ and ${K}$ is nonsingular.
Then ${\B}$ is an affine subspace of ${\R^m\times\R^n}$ of dimension ${m}$.
\end{lemma}

\subsubsection{Interconnection and well-posedness}\label{ssec:interconnection}

Consider ${k}$ open quadratic optimization programs with signal spaces ${\W_i}$ and behaviors ${\B_i}$, ${i=1,\ldots,k}$.
Set ${\W:=\W_1\times\cdots\times\W_k}$ and let ${\pi_i:\W\to\W_i}$ denote the projection onto the variables ${w_i=(u_i,x_i)}$ of the ${i}$-th program.
In behavioral systems theory, interconnection is variable sharing~\cite{Willems2007,PolWil1998}.
Motivated by this paradigm, we model an interconnection of open \glspl{QP} by an affine subspace ${\mathcal{C}\subseteq\W}$ defined by equations identifying variables of distinct programs, and define their \emph{interconnected behavior} as
\begin{equation}\label{eq:interconnected}
\B:=\mathcal{C}\cap\bigcap_{i=1}^{k}\pi_i^{-1}(\B_i).
\end{equation}
The interconnection is \emph{well posed} if ${\B}$ is a singleton.
If ${P_i\succeq0}$ and ${K_i}$ is nonsingular for every ${i}$, then by Lemma~\ref{lem:affine} each ${\B_i}$ is an affine subspace, so that there exist matrices ${D_i}$ and vectors ${d_i}$ such that ${\B_i=\{w_i\in\W_i\mid D_iw_i=d_i\}}$, and likewise there exist ${D_0}$ and ${d_0}$ such that ${\mathcal{C}=\{w\in\W\mid D_0w=d_0\}}$.
Since each ${\pi_i}$ is linear, ${\B=\{w\in\W\mid Dw=d\}}$ with
\begin{subequations}\label{eq:stacked}
\begin{align}
D&:=\begin{bmatrix}D_0\tp&(D_1\pi_1)\tp&\cdots&(D_k\pi_k)\tp\end{bmatrix}\tp,\\
d&:=\begin{bmatrix}d_0\tp&d_1\tp&\cdots&d_k\tp\end{bmatrix}\tp.
\end{align}
\end{subequations}
Let ${R_\star}$ be the matrix selecting a block ${w_\star=R_\star w}$ of shared variables. We say that the interconnection has a \emph{single loop} if there exist matrices ${E}$, ${\Xi}$ and vectors ${e}$, ${b_\star}$, with ${R_\star E=I}$ and ${R_\star e=0}$, such that
${\B=\{\,Ez+e\mid z=\Xi z+b_\star\,\}. }$
We refer to ${\Xi}$ as the \emph{loop map}. Since ${R_\star(Ez+e)=z}$, the map ${z\mapsto Ez+e}$ is a bijection from the solution set of the loop equation ${z=\Xi z+b_\star}$ onto ${\B}$. Throughout the rest of the paper, ${\spec(X)}$ denotes the spectrum of a square
matrix ${X}$. For a partition ${w=(w_{\mathrm{in}},w_{\mathrm{out}})}$ into inputs
and outputs, write ${D=\begin{bmatrix}D_{\mathrm{in}}&D_{\mathrm{out}}\end{bmatrix}}$
for the corresponding column partition.

\begin{proposition}[Closure and well-posedness]\label{thm:wp}
Consider ${k}$ open quadratic optimization programs with matrices ${P_i}$ and ${K_i}$ of the form~\eqref{eq:oqp} and~\eqref{eq:kkt}, ${i=1,\dots,k}$, behaviors ${\B_1,\dots,\B_k}$, an interconnection ${\mathcal{C}}$, the interconnected behavior ${\B}$ in~\eqref{eq:interconnected}, the matrices ${D}$ and ${d}$ in~\eqref{eq:stacked}, and a partition ${w=(w_{\mathrm{in}},w_{\mathrm{out}})}$ of the variables, with ${D=\begin{bmatrix}D_{\mathrm{in}}&D_{\mathrm{out}}\end{bmatrix}}$.
Assume ${P_i\succeq0}$ and ${K_i}$ is nonsingular for ${i=1,\dots,k}$.
Then the following statements hold.
\begin{enumerate}
\item[(i)] ${\B}$ is the behavior of an open quadratic optimization program with input ${w_{\mathrm{in}}}$, decision variable ${w_{\mathrm{out}}}$, and nonsingular \gls{KKT} matrix if and only if ${\B\neq\emptyset}$, ${D_{\mathrm{out}}}$ has full column rank, and ${\im D_{\mathrm{in}}\subseteq\im D_{\mathrm{out}}}$.
\item[(ii)] ${\B}$ is a singleton if and only if ${Dw=d}$ is consistent and ${D}$ has full column rank.
\item[(iii)] If the interconnection has a single loop, with loop map ${\Xi}$ and vector ${b_\star}$, then ${\B}$ is a singleton if and only if ${1\notin\spec(\Xi)}$, in which case ${w_\star=(I-\Xi)^{-1}b_\star}$ and ${{w=Ew_\star+e}.}$
\end{enumerate}
\end{proposition}

\subsubsection{Computation by sequential module evaluation}
When the modules are implemented separately and communicate only through their interfaces, it is natural to evaluate them one at a time in a fixed order, each using the most recent values of the shared variables.
We analyze this scheme for two open quadratic optimization programs interconnected by adding their objectives, as in~\eqref{eq:composed_problem}, which amounts to sharing an input and its conjugate variable.
The scheme is then a block Gauss--Seidel iteration~\cite{Varga2000,GolubVanLoan2013}, and the results below quantify when it converges, at what rate, and the error incurred after a single round of communication.

Consider two open quadratic optimization programs, the first with input ${\xi}$ and the second with decision variable ${\xi}$, and the program obtained by adding their objectives and imposing the constraints of both,
\begin{equation}\label{eq:joint}
\begin{array}{cl}
\displaystyle\min_{x,\,\xi} & \tfrac12x\tp P_1x+(q_1+N_1\xi)\tp x+\tfrac12\xi\tp P_2\xi+q_2\tp\xi\\[6pt]
\textup{s.t.} & A_1x+M_1\xi=b_1,\qquad A_2\xi=b_2,
\end{array}
\end{equation}
in the variables ${x\in\R^{n_1}}$ and ${\xi\in\R^{n_2}}$.
With ${w_1:=(x,\lambda_1)}$ and ${w_2:=(\xi,\lambda_2)}$, the \gls{KKT} system of~\eqref{eq:joint} is
\begin{equation}\label{eq:block}
\begin{bmatrix}K_1&L\\L\tp&K_2\end{bmatrix}
\begin{bmatrix}w_1\\w_2\end{bmatrix}
=\begin{bmatrix}r_1\\r_2\end{bmatrix},
\end{equation}
with
\begin{equation}\label{eq:CZ}
\begin{aligned}
K_i&:=\begin{bmatrix}P_i&A_i\tp\\A_i&0\end{bmatrix},&
r_i&:=\begin{bmatrix}-q_i\\b_i\end{bmatrix},\\
L&:=\begin{bmatrix}C&0\end{bmatrix},&
C&:=\begin{bmatrix}N_1\\M_1\end{bmatrix}.
\end{aligned}
\end{equation}
Let ${y_1:=N_1\tp x+M_1\tp\lambda_1}$ be the conjugate variable of the input of the first program, so that the stationarity condition of~\eqref{eq:joint} in ${\xi}$ reads ${y_1+P_2\xi+q_2+A_2\tp\lambda_2=0}$.
When the modules are implemented separately, the first program receives ${\xi}$ and returns ${y_1}$, and the second program receives ${y_1}$ and solves ${\min_{\xi}\{\tfrac12\xi\tp P_2\xi+(q_2+y_1)\tp\xi : A_2\xi=b_2\}}$, whose \gls{KKT} system is ${K_2w_2=r_2-L\tp w_1}$.
Evaluating the second program at the most recent ${y_1}$ and then the first at the resulting ${\xi}$ gives the block Gauss--Seidel iteration
\begin{equation}\label{eq:iter}
\begin{aligned}
w_2^{(j+1)}&=K_2^{-1}\big(r_2-L\tp w_1^{(j)}\big),\\
w_1^{(j+1)}&=K_1^{-1}\big(r_1-Lw_2^{(j+1)}\big),
\end{aligned}
\end{equation}
and the \emph{interaction matrices} of~\eqref{eq:block} are
\begin{equation}\label{eq:gamma}
\Gamma:=K_2^{-1}L\tp K_1^{-1}L,
\qquad
\widetilde\Gamma:=K_1^{-1}LK_2^{-1}L\tp .
\end{equation}
In the following statement, ${\|\cdot\|}$ denotes a vector norm and the matrix norm it induces, and ${\rho(X)}$ the spectral radius of a square matrix ${X}$.

\begin{proposition}[Module-wise computation]\label{prop:gs}
Consider the system~\eqref{eq:block} with the matrices~\eqref{eq:CZ}, the iteration~\eqref{eq:iter}, and the matrices~\eqref{eq:gamma}.
Assume ${K_1}$ and ${K_2}$ are nonsingular.
Then the following statements hold.
\begin{enumerate}
\item[(i)] The system~\eqref{eq:block} has a unique solution if and only if ${1\notin\spec(\Gamma)}$.
\item[(ii)] The spectral radii of ${\Gamma}$ and ${\widetilde\Gamma}$ coincide: ${\rho(\Gamma)=\rho(\widetilde\Gamma)}$.
\item[(iii)] If ${1\notin\spec(\Gamma)}$ and ${w^\star}$ denotes the solution of~\eqref{eq:block}, the iterates of~\eqref{eq:iter} satisfy
\[
\begin{aligned}
w_1^{(j+1)}-w_1^\star&=\widetilde\Gamma\,(w_1^{(j)}-w_1^\star),\\
w_2^{(j+1)}-w_2^\star&=-K_2^{-1}L\tp(w_1^{(j)}-w_1^\star).
\end{aligned}
\]
The iteration then converges to ${w^\star}$ from every initialization ${w_1^{(0)}}$ if and only if ${\rho(\Gamma)<1}$.
Moreover,
\[
\|w_1^{(j)}-w_1^\star\|\le\|\widetilde\Gamma^{\,j}\|\,\|w_1^{(0)}-w_1^\star\|,
\]
and ${\|\widetilde\Gamma^{\,j}\|^{1/j}\to\rho(\Gamma)}$ as ${j\to\infty}$.
\end{enumerate}
\end{proposition}

\noindent
We now quantify the error after one iteration.
Let the columns of ${Z}$ form a basis of ${\ker A_2}$ and let ${\widehat P_2:=Z\tp P_2Z}$ be the \emph{reduced Hessian} of the second program, that is, the Hessian of its objective on the tangent space of its constraints~\cite{NocedalWright2006}.
Let ${V_1}$ be the value function of the first program, let ${\xi^\star}$ be the primal component of ${w_2^\star}$, and let ${\xi^{(1)}}$ be the primal component of ${w_2^{(1)}}$ produced by~\eqref{eq:iter} from ${w_1^{(0)}=0}$.
The solution of~\eqref{eq:block} is the \gls{KKT} point of~\eqref{eq:joint}, and it is a minimizer whenever the objective of~\eqref{eq:joint} is convex on the affine set defined by the constraints.

\begin{lemma}[Error analysis]\label{lem:error}
Consider the program~\eqref{eq:joint}, its \gls{KKT} system~\eqref{eq:block} with the matrices~\eqref{eq:CZ}, the iteration~\eqref{eq:iter} initialized at ${w_1^{(0)}=0}$, and the reduced Hessian ${\widehat P_2}$.
Assume ${P_i\succeq0}$ and ${K_i}$ nonsingular for ${i=1,2}$, and ${1\notin\spec(\Gamma)}$.
Then ${\widehat P_2\succ0}$ and
\begin{equation}\label{eq:error}
\begin{aligned}
\xi^{(1)}-\xi^\star&=Z\widehat P_2^{-1}Z\tp\,\nabla V_1(\xi^\star),\\
\nabla V_1(\xi^\star)&=N_1\tp x^\star+M_1\tp\lambda_1^\star .
\end{aligned}
\end{equation}
Moreover, ${\xi^{(1)}=\xi^\star}$ for all ${q_1,q_2,b_1,b_2}$ if and only if ${CZ=0}$, that is, ${\ker A_2\subseteq\ker N_1\cap\ker M_1}$.
\end{lemma}

\subsubsection{The control architecture}\label{ssec:stack-solved}

We return to the first question of Section~\ref{ssec:control_architecture}.
Assume the \gls{KKT} matrices of the modules~\eqref{eq:estimation_cls}--\eqref{eq:control_cls} are nonsingular.
By Lemma~\ref{lem:affine}, the shared variables satisfy affine relations, which together with the response of~\eqref{eq:system_linear} over the horizon we write as
\begin{subequations}\label{eq:maps}
\begin{align}
\hat x_1&=S_{\mathrm{est}}^{x}\tilde x+S_{\mathrm{est}}^{u}u+s_{\mathrm{est}},\\
x_{\mathrm{ref}}&=S_{\mathrm{pl}}^{x}\hat x_1+S_{\mathrm{pl}}^{g}x_{\mathrm{goal}}+s_{\mathrm{pl}},\\
u&=S_{\mathrm{ctrl}}^{x}\hat x_1+S_{\mathrm{ctrl}}^{r}x_{\mathrm{ref}}+s_{\mathrm{ctrl}},\\
\tilde x&=S_{\mathrm{sys}}u+s_{\mathrm{sys}}.
\end{align}
\end{subequations}
The behavior of the architecture is the solution set of~\eqref{eq:maps}, that is, the intersection~\eqref{eq:interconnected} of the four relations.
Eliminating ${\tilde x}$, ${\hat x_1}$, and ${x_{\mathrm{ref}}}$, which are affine functions of ${u}$, gives ${u=\Xi u+b_\star}$, with
\[
\Xi:=(S_{\mathrm{ctrl}}^{x}+S_{\mathrm{ctrl}}^{r}S_{\mathrm{pl}}^{x})(S_{\mathrm{est}}^{x}S_{\mathrm{sys}}+S_{\mathrm{est}}^{u}),
\]
so that the interconnection has a single loop with ${w_\star=u}$.
By Proposition~\ref{thm:wp}, the architecture is well posed if and only if ${1\notin\spec(\Xi)}$, in which case ${u^\star=(I-\Xi)^{-1}b_\star}$ and the remaining variables follow from~\eqref{eq:maps}.
Evaluating the four relations in the order of the loop is the iteration ${u^{(j+1)}=\Xi u^{(j)}+b_\star}$, which converges to ${u^\star}$ from every initialization if and only if ${\rho(\Xi)<1}$.
\section{Primal--dual description}\label{sec:pd}

The graph of the solution map~\eqref{eq:solmap} captures the optimizer of a program, but not its multiplier.
It is therefore inadequate at capturing interconnections such as~\eqref{eq:port_interconnection}, which constrain primal and dual variables jointly.

\subsection{The primal--dual behavior}\label{ssec:pdbehavior}

Consider the open \gls{QP}~\eqref{eq:oqp} with ${K}$ nonsingular, and let ${x^\star(u)}$ and ${\lambda^\star(u)}$ denote the primal and dual components of the solution of~\eqref{eq:kkt}.
The conjugate variable of the input is
\begin{equation}\label{eq:rho}
y:=N\tp x^\star(u)+M\tp\lambda^\star(u).
\end{equation}
We associate with~\eqref{eq:oqp} the signal space ${\W^{\mathrm{pd}}:=\R^m\times\R^m}$ and the \emph{primal--dual behavior}
\begin{equation}\label{eq:pdbehavior}
\B^{\mathrm{pd}}:=\{\,(u,y)\in\W^{\mathrm{pd}}\mid y=N\tp x^\star(u)+M\tp\lambda^\star(u)\,\},
\end{equation}
the graph of the map from the input to its conjugate variable.
Consider the resistive network of Section~\ref{ssec:resistive_networks} with prescribed internal injections ${c_{\mathrm{i}}}$ and prescribed external potentials ${v_{\mathrm{e}}}$, that is, the open program ${\min_{f\in\R^m}\{\tfrac12f\tp Rf+v_{\mathrm{e}}\tp B_{\mathrm{e}}f : B_{\mathrm{i}}f=c_{\mathrm{i}}\}}$ with input ${v_{\mathrm{e}}}$.
Its \gls{KKT} conditions are the first equation in~\eqref{eq:thomson_kkt} together with ${B_{\mathrm{i}}f=c_{\mathrm{i}}}$, with ${\lambda=v_{\mathrm{i}}}$, and its \gls{KKT} matrix is nonsingular, since ${B_{\mathrm{i}}}$ has full row rank when ${\mathbb{G}}$ is connected and has at least one external node.
Its boundary relation~\eqref{eq:boundary_relation} is exactly~\eqref{eq:pdbehavior}, with ${u=v_{\mathrm{e}}}$ and ${y=c_{\mathrm{e}}}$.

We now recall two standard notions from symplectic geometry~\cite{arnold1989}.
The \emph{standard symplectic form} on ${\R^m\times\R^m}$ is ${\omega((u_1,y_1),(u_2,y_2)):=y_2\tp u_1-y_1\tp u_2}$, and a subspace of ${\R^m\times\R^m}$ is \emph{Lagrangian} if it has dimension ${m}$ and ${\omega}$ vanishes on it.
With ${C:=\begin{bmatrix}N\tp&M\tp\end{bmatrix}\tp}$, the matrix ${-C\tp K^{-1}C}$ is the Schur complement of ${K}$ in ${\big[\begin{smallmatrix}K&C\\C\tp&0\end{smallmatrix}\big]}$~\cite{zhang2005schur}.

\begin{theorem}[Primal--dual behavior]\label{thm:pd}
Consider the open quadratic optimization program~\eqref{eq:oqp}, with \gls{KKT} matrix ${K}$, the matrix ${C}$ defined above, value function ${V}$, and primal--dual behavior~\eqref{eq:pdbehavior}.
Assume ${P\succeq0}$ and ${K}$ is nonsingular.
Then the following statements hold.
\begin{enumerate}
\item[(i)] ${V}$ is a quadratic function of ${u}$ with ${\nabla V(u)=y}$, so that ${\B^{\mathrm{pd}}}$ is the graph of ${\nabla V}$, and ${y=Su+s}$ with
\begin{equation}\label{eq:schur}
S=\nabla^2V=-\,C\tp K^{-1}C .
\end{equation}
\item[(ii)] The linear part of ${\B^{\mathrm{pd}}}$, that is ${\{(u,Su):u\in\R^m\}}$, is a Lagrangian subspace of ${\R^m\times\R^m}$.
\item[(iii)] If the input affects~\eqref{eq:oqp} through the constraints, then ${V}$ is convex and ${S\succeq0}$.
If the input affects~\eqref{eq:oqp} through the objective, then ${V}$ is concave and ${S\preceq0}$.
In general, ${S}$ need not be sign definite.
\end{enumerate}
\end{theorem}
\noindent
Theorem~\ref{thm:pd} recovers the Lagrangian relation between
terminal currents and potentials that underlies the compositional
theory of passive linear networks~\cite{BaezFong2018}. For the network of Section~\ref{ssec:resistive_networks}, the symmetry of ${S}$ is the reciprocity of resistive multiports~\cite{AndersonVongpanitlerd1973}.

\subsection{Interconnection of primal--dual behaviors}\label{ssec:portint}

Consider two open quadratic optimization programs with matrices ${P_i}$ and ${K_i}$ as in~\eqref{eq:oqp} and~\eqref{eq:kkt} and value functions ${V_i(u_i,c_i)}$, ${i=1,2}$, where ${c_1,c_2\in\R^{\kappa}}$ are shared variables, and let ${c_{\mathrm{s}}\in\R^{\kappa}}$ be given.
The \emph{port interconnection} of the two programs is the pair of constraints
\begin{equation}\label{eq:portgen}
c_1+c_2=c_{\mathrm{s}},
\qquad
\nabla_{c_1}V_1(u_1,c_1)=\nabla_{c_2}V_2(u_2,c_2).
\end{equation}
The corresponding \emph{interconnected program} minimizes the sum of the objectives of the two programs over their decision variables and over ${(c_1,c_2)}$, subject to the constraints of both programs and to ${c_1+c_2=c_{\mathrm{s}}}$, and its value function ${V_{12}}$ depends on ${u:=(u_1,u_2)}$.
The \emph{partial infimal convolution} of ${V_1}$ and ${V_2}$ over the shared variables~\cite{rockafellar1970convex} is defined as
\begin{equation}\label{eq:infconv}
(V_1\,\Box\,V_2)(u):=\inf_{c_1+c_2=c_{\mathrm{s}}}\big\{V_1(u_1,c_1)+V_2(u_2,c_2)\big\}.
\end{equation}
If ${P_i\succeq0}$ and ${K_i}$ is nonsingular for ${i=1,2}$, by Theorem~\ref{thm:pd} the function ${(u,c_1)\mapsto V_1(u_1,c_1)+V_2(u_2,c_{\mathrm{s}}-c_1)}$ is quadratic, and we denote its Hessian by
\[
H:=\begin{bmatrix}H_{uu}&H_{uc}\\H_{cu}&H_{cc}\end{bmatrix},
\]

\begin{theorem}[Port interconnection]\label{thm:port}
Consider two open quadratic optimization programs with matrices ${P_i}$ and ${K_i}$ as in~\eqref{eq:oqp} and~\eqref{eq:kkt}, value functions ${V_1}$ and ${V_2}$, shared variables ${c_1,c_2\in\R^{\kappa}}$, the port interconnection~\eqref{eq:portgen}, the corresponding interconnected program with value function ${V_{12}}$, and the matrix ${H}$ defined above.
Assume that
\begin{enumerate}
\item[(H1)] ${P_i\succeq0}$ and ${K_i}$ is nonsingular for ${i=1,2}$.
\item[(H2)] the input ${(u_i,c_i)}$ affects the ${i}$-th program through the constraints, ${i=1,2}$.
\item[(H3)] ${H_{cc}\succ0}$.
\end{enumerate}
Then the following statements hold.
\begin{enumerate}
\item[(i)] ${V_{12}=V_1\,\Box\,V_2}$, and the infimum in~\eqref{eq:infconv} is attained exactly at the unique ${(c_1,c_2)}$ satisfying~\eqref{eq:portgen}.
\item[(ii)] ${(x_1,x_2,c_1,c_2)}$ is optimal for the interconnected program if and only if ${(c_1,c_2)}$ satisfies~\eqref{eq:portgen} and ${x_i}$ is optimal for the ${i}$-th program at ${c_i}$, ${i=1,2}$.
\item[(iii)] ${V_{12}}$ is quadratic, with Hessian ${H_{uu}-H_{uc}H_{cc}^{-1}H_{cu}}$, the Schur complement of ${H_{cc}}$ in ${H}$.
\end{enumerate}
\end{theorem}
\noindent
Interconnecting three or more programs amounts to repeated Schur complementation, whose result does not depend on the order of elimination whenever the intermediate pivots are nonsingular, by the quotient property of Schur complements~\cite[Ch.~1]{zhang2005schur}. In the language of circuit theory, Kron reduction may be performed one interior node at a time~\cite{DorflerBullo2013}.

\begin{example}[Series law]\label{ex:series}
Consider two edges of resistances ${r_1,r_2>0}$ connected in series, with end potentials ${v_{\mathrm L},v_{\mathrm R}\in\R}$, shared potential ${v\in\R}$, and ${\Pi:=\big[\begin{smallmatrix*}[r]1&-1\\-1&1\end{smallmatrix*}\big]}$.
With the potentials as inputs, the value function of an edge of resistance ${r}$ is ${W_r(v_1,v_2)=-(v_1-v_2)^2/(2r)}$, and its gradient ${-\tfrac1r\Pi(v_1,v_2)}$ is the boundary relation~\eqref{eq:boundary_relation} of the edge.
At the shared node, the interconnection~\eqref{eq:port_interconnection} requires the currents of the two edges to sum to zero, that is, ${\partial_v\big(W_{r_1}(v_{\mathrm L},v)+W_{r_2}(v,v_{\mathrm R})\big)=0}$.
The function in brackets is strictly concave in ${v}$ and is maximized at ${v^\star=(r_2v_{\mathrm L}+r_1v_{\mathrm R})/(r_1+r_2)}$, where it equals ${-(v_{\mathrm L}-v_{\mathrm R})^2/(2(r_1+r_2))}$.
The gradient of this maximum with respect to ${(v_{\mathrm L},v_{\mathrm R})}$ is ${c_{\mathrm e}=-\tfrac{1}{r_1+r_2}\Pi\,v_{\mathrm e}}$, the series law.
\end{example}

\section{Value description}\label{sec:prob}

The primal--dual behavior~\eqref{eq:pdbehavior} captures the gradient of the value function, and hence determines ${V}$ only up to an additive constant.
The description of this section augments it with the optimal value, and stands to Section~\ref{sec:pd} as contact geometry stands to symplectic geometry.

\subsection{Entropic regularization and the free energy}\label{ssec:lift}

Consider an open optimization program~\eqref{eq:open}, with objective ${f}$, and, for ${T>0}$, replace the minimization of ${f}$ by the minimization of ${\E_p[f]+T\,\E_p[\log p]}$ over probability densities ${p}$ on the feasible set.
This is the \emph{entropic regularization} of the program~\cite{WainwrightJordan2008TR}.
Its minimizer is proportional to ${\exp(-f/T)}$ and its minimum is the free energy ${-T\log\int\exp(-f/T)}$ of statistical mechanics~\cite{WainwrightJordan2008TR}, which equals the log-partition function of an exponential family up to sign and scale~\cite{WainwrightJordan2008TR}.
For open \glspl{QP} with nonsingular \gls{KKT} matrices, the mean of the minimizing density is the minimizer of the program for every ${T>0}$, under some assumptions, as we now discuss.

Consider the open \gls{QP}~\eqref{eq:oqp}, with fixed ${u\in\R^m}$ and ${b\in\R^p}$, and write ${V(u,b)}$ and ${x^\star(u,b)}$ to make the dependence on ${b}$ explicit.
Let ${d:=\dim\ker A}$ and let the columns of ${Z}$ form an orthonormal basis of ${\ker A}$, so that ${\widehat P:=Z\tp PZ}$ is the reduced Hessian of~\eqref{eq:oqp}.
For ${x_0\in\mathcal X(u)}$, the integral of a function ${g}$ over the affine subspace ${\mathcal{X}(u)}$ is defined as ${\int_{\mathcal X(u)}g(x)\,dx:=\int_{\R^d}g(x_0+Z\eta)\,d\eta}$, which does not depend on ${x_0}$ and ${Z}$, since two such choices differ by a translation and an orthogonal change of coordinates of ${\R^d}$.
For ${T>0}$ and ${f}$ the objective of~\eqref{eq:oqp}, define the \emph{free energy}
\begin{equation}\label{eq:freeenergy}
\mathcal{F}_T(u,b):=-T\log\int_{\mathcal{X}(u)}\exp\big(-f(x,u)/T\big)\,dx ,
\end{equation}
and, when the integral in~\eqref{eq:freeenergy} is finite, let ${p_T}$ be the probability density on ${\mathcal X(u)}$ proportional to ${\exp(-f(x,u)/T)}$.

\begin{proposition}[Free energy]\label{prop:lift}
Consider the open \gls{QP}~\eqref{eq:oqp}, the free energy~\eqref{eq:freeenergy}, and the density ${p_T}$.
Assume ${P\succeq0}$ and ${K}$ is nonsingular.
Then, for every ${T>0}$, the following statements hold.
\begin{enumerate}
\item[(i)] The integral in~\eqref{eq:freeenergy} is finite, and ${p_T}$ is the Gaussian density with mean ${x^\star(u,b)}$, the minimizer of~\eqref{eq:oqp}, and covariance ${T\,Z\widehat P^{-1}Z\tp}$.
\item[(ii)] The free energy satisfies
\[
\mathcal{F}_T(u,b)=V(u,b)+\tfrac{T}{2}\log\det\big(\tfrac{1}{2\pi T}\widehat P\big).
\]
\item[(iii)] ${\nabla_b\mathcal{F}_T=\nabla_bV=-\lambda^\star}$ and ${\nabla_u\mathcal{F}_T=\nabla_uV=y}$.
\end{enumerate}
\end{proposition}
\noindent
Proposition~\ref{prop:lift} states that the regularized density of an open \gls{QP} is Gaussian and centered at the minimizer for every ${T}$, and that the free energy differs from the value function by a term that depends on the data only through the reduced Hessian, so that the two functions have the same gradients.

\subsection{The value behavior}\label{ssec:valuebehavior}

By Proposition~\ref{prop:lift}, the free energy carries the value of an open \gls{QP} together with its conjugate variable, which the primal--dual behavior alone does not.
Let ${\W^{\mathrm{v}}:=\R^m\times\R^m\times\R}$.
In analogy with~\eqref{eq:pdbehavior}, we associate with~\eqref{eq:oqp} the \emph{value behavior}
\begin{equation}\label{eq:jet}
\B^{\mathrm{v}}:=\{\,(u,y,\phi)\in\W^{\mathrm{v}}\mid y=\nabla V(u),\ \phi=V(u)\,\}.
\end{equation}
The \emph{standard contact form} on ${\W^{\mathrm{v}}}$ is ${\vartheta:=d\phi-y\tp du}$, and a submanifold of dimension ${m}$ on which ${\vartheta}$ vanishes is \emph{Legendrian}~\cite{arnold1989}.
Up to the ordering of the coordinates, ${\B^{\mathrm{v}}}$ is the image of the ${1}$-jet extension ${u\mapsto(u,V(u),\nabla V(u))}$ of ${V}$, hence a Legendrian submanifold of ${(\W^{\mathrm{v}},\vartheta)}$~\cite{arnold1989,Limouzineau2016}.
In the notation of Section~\ref{sec:apps}, with ${\theta}$ in place of ${u}$, the value behavior of a potential ${\Phi}$ is the set of triples ${(\theta,\nabla\Phi(\theta),\Phi(\theta))}$.
Accordingly, the interconnection of two value behaviors identifies the inputs, adds the conjugate variables, and adds the values,
\begin{equation}\label{eq:valadd}
u_1=u_2,
\qquad
y=y_1+y_2,
\qquad
\phi=\phi_1+\phi_2,
\end{equation}
the last being the rule absent from Section~\ref{sec:pd}.

The free energies of Section~\ref{sec:apps} are defined as follows.
Let ${z\in\R^{\ell}}$ be a latent variable and let ${\psi_1,\psi_2:\R^{\ell}\to\R_{>0}}$ be measurable functions, called \emph{factors}, each associated with one module and modeling the prior and likelihood terms known to it.
Let ${Z_{12}:=\int_{\R^\ell}\psi_1(z)\psi_2(z)\,dz}$ and, when ${Z_{12}}$ is positive and finite, let ${p_{12}:=\psi_1\psi_2/Z_{12}}$ be the normalized product of the factors.
When ${\psi_1\psi_2}$ is the joint density of the data and of the latent variable, ${Z_{12}}$ is the evidence of the model~\cite{Blei2017}.
Let ${\mu}$ be a probability density on ${\R^{\ell}}$, called a \emph{belief}, let ${\Ent(\mu):=-\E_\mu[\log\mu]}$ be its differential entropy~\cite{CoverThomas2006}, and let ${\mathcal{F}_i:=\E_\mu[\log \mu-\log\psi_i]}$, ${i=1,2}$, and ${\mathcal{F}_{12}:=\E_\mu[\log \mu-\log(\psi_1\psi_2)]}$ be the variational free energies of the factors and of their product~\cite{WainwrightJordan2008TR,Blei2017}.
Whenever these expectations are finite and ${Z_{12}}$ is positive and finite, one has ${\mathcal F_{12}=\KL(\mu\,\|\,p_{12})-\log Z_{12}}$, so that the minimum of ${\mathcal F_{12}}$ over any family of beliefs containing ${p_{12}}$ is ${-\log Z_{12}}$, the variational characterization of the evidence~\cite{Blei2017}.

Finally, consider the program~\eqref{eq:joint} in which ${\xi}$ affects the first program through the objective, that is, ${M_1=0}$.
Let ${Z_1}$ and ${Z_2}$ have orthonormal columns spanning ${\ker A_1}$ and ${\ker A_2}$, and let ${\widehat P_i:=Z_i\tp P_iZ_i}$ and ${\widehat L:=Z_1\tp N_1Z_2}$.
The columns of ${\diag(Z_1,Z_2)}$ span the kernel of the constraint matrix ${\diag(A_1,A_2)}$ of~\eqref{eq:joint}, whose reduced Hessian is therefore ${\widehat P_{12}:=\big[\begin{smallmatrix}\widehat P_1&\widehat L\\ \widehat L\tp&\widehat P_2\end{smallmatrix}\big]}$.

\begin{theorem}[Composition of values]\label{thm:value}
Consider two open \glspl{QP} with matrices ${P_i}$ and ${K_i}$ of the form~\eqref{eq:oqp} and~\eqref{eq:kkt}, value functions ${V_i}$, and value behaviors ${\B_i^{\mathrm{v}}}$, ${i=1,2}$, the factors ${\psi_1,\psi_2}$, the belief ${\mu}$, the free energies ${\mathcal{F}_1,\mathcal{F}_2,\mathcal{F}_{12}}$, and the program~\eqref{eq:joint} with ${\xi}$ affecting the first program through the objective and the matrices ${\widehat P_{12}}$, ${\widehat P_1}$, ${\widehat P_2}$, ${\widehat L}$ defined above.
Assume ${P_i\succeq0}$ and ${K_i}$ nonsingular for ${i=1,2}$, ${Z_{12}}$ positive and finite, and ${\E_\mu[\,|\log\mu|+|\log\psi_1|+|\log\psi_2|\,]<\infty}$.
Then the following statements hold.
\begin{enumerate}
\item[(i)] Adding constants ${\kappa_1,\kappa_2\in\R}$ to the objectives of the two programs leaves ${\B_1^{\mathrm{pd}}}$ and ${\B_2^{\mathrm{pd}}}$ unchanged and adds ${\kappa_1+\kappa_2}$ to the value of any program minimizing the sum of the two objectives over a fixed feasible set.
\item[(ii)] The interconnection~\eqref{eq:valadd} maps ${(\B_1^{\mathrm{v}},\B_2^{\mathrm{v}})}$ to the value behavior of ${V_1+V_2}$.
\item[(iii)] If ${\widehat P_{12}\succ0}$, then
\[
\log\det\widehat P_{12}=\log\det\widehat P_1+\log\det(\widehat P_2-\widehat L\tp\widehat P_1^{-1}\widehat L).
\]
\item[(iv)] ${\mathcal{F}_1+\mathcal{F}_2=\mathcal{F}_{12}-\Ent(\mu)}$.
\end{enumerate}
\end{theorem}
\noindent
By Theorem~\ref{thm:value}, the value of an interconnection is not determined by the primal--dual behaviors of its modules, and ${\mathcal F_{12}=\mathcal F_1+\mathcal F_2+\Ent(\mu)}$, so adding the two free energies counts the entropy of the shared belief twice.
Equivalently, ${\widetilde{\mathcal F}_1+\mathcal F_2=\mathcal F_{12}}$, with ${\widetilde{\mathcal F}_1:=\mathcal F_1+\Ent(\mu)=\E_\mu[-\log\psi_1]}$, that is, the free energies compose additively once one module transmits the expected negative logarithm of its factor without the entropy of the belief.

\begin{remark}[Gaussian architectures]\label{rem:cca}
Consider~\eqref{eq:joint} with ${A_1,A_2}$ absent, so that ${K_i=P_i}$ and ${L=N_1}$ in~\eqref{eq:block}, and a Gaussian random vector ${(X_1,X_2)}$ with values in ${\R^{n_1}\times\R^{n_2}}$ whose covariance matrix is the inverse of ${\Lambda:=\big[\begin{smallmatrix}P_1&L\\L\tp&P_2\end{smallmatrix}\big]}$, called its precision matrix.
Assume ${\Lambda\succ0}$, and let ${k:=\min\{n_1,n_2\}}$ and ${R:=P_1^{-1/2}LP_2^{-1/2}}$.
Then the singular values ${s_1,\dots,s_k}$ of ${R}$ are the canonical correlations of ${(X_1,X_2)}$~\cite{hotelling1936}, the eigenvalues of ${\Gamma}$ are ${s_1^2,\dots,s_k^2}$ together with ${n_2-k}$ zeros, and the mutual information of ${X_1}$ and ${X_2}$ is ${-\tfrac12\log\det(I-\Gamma)}$~\cite{CoverThomas2006}.
Indeed, the precision matrix of ${(P_1^{1/2}X_1,P_2^{1/2}X_2)}$ is ${\big[\begin{smallmatrix}I&R\\R\tp&I\end{smallmatrix}\big]}$, a singular value decomposition of ${R}$ splits this vector into ${k}$ independent pairs with precision matrices ${\big[\begin{smallmatrix}1&s_i\\s_i&1\end{smallmatrix}\big]}$, hence with correlation ${-s_i}$, and into independent remaining coordinates, and ${\Gamma=P_2^{-1/2}R\tp RP_2^{1/2}}$.
By Lemma~\ref{lem:error}, ${\xi^{(1)}=\xi^\star}$ for all ${q_1}$ and ${q_2}$ if and only if ${L=0}$, that is, if and only if ${X_1}$ and ${X_2}$ are independent.
\end{remark}

\section{Conclusions}\label{sec:concl}

\begin{table}[t]
\centering
\caption{Descriptions of an open \gls{QP}.}
\label{tab:summary}
\renewcommand{\arraystretch}{1.25}
\footnotesize
\setlength{\tabcolsep}{4pt}
\begin{tabular}{@{}lccc@{}}
\hline
& \textbf{primal} & \textbf{primal--dual} & \textbf{value}\\
\hline
transmits & ${x^\star}$ & ${(u,\nabla V)}$ & ${(u,\nabla V,V)}$\\
behavior & affine & Lagrangian & Legendrian\\
interconnection & \eqref{eq:interconnected} & \eqref{eq:portgen} & \eqref{eq:portgen}, \eqref{eq:valadd}\\
application & control stack & resistive networks & belief networks\\
\hline
\end{tabular}
\end{table}

We have studied open quadratic optimization programs as systems interconnected by variable sharing. The analysis characterizes well-posedness and, for modules interconnected by adding their objectives, the convergence of sequential module evaluation, and develops interconnection rules for primal, primal--dual, and value behaviors. These descriptions distinguish the information needed to compose optimizers, conjugate variables, and optimal values, as illustrated by control architectures, resistive networks, and variational inference. Extending the primal--dual description to subdifferential relations of convex value functions and exploring compositional approaches based on category theory are natural directions for future work.

\appendix

\begin{proof}[Proof of Lemma~\ref{lem:affine}]
Since ${P\succeq0}$ and the constraints are affine, the \gls{KKT} conditions~\eqref{eq:kkt} are necessary and sufficient for optimality~\cite[Sec.~10.1.1]{BoydVandenberghe2004}. Since ${K}$ is nonsingular, for every ${u\in\R^m}$ the system~\eqref{eq:kkt} has the unique solution
\[
\begin{bmatrix}x^\star(u)\\\lambda^\star(u)\end{bmatrix}
=K^{-1}\begin{bmatrix}-q\\b\end{bmatrix}-K^{-1}\begin{bmatrix}N\\M\end{bmatrix}u ,
\]
so that ${X^\star(u)=\{x^\star(u)\}}$ and ${x^\star}$ is an affine function of ${u}$. Hence, ${\B=\graph X^\star}$ is the image of the injective affine map ${u\mapsto(u,x^\star(u))}$, that is, an affine subspace of ${\R^m\times\R^n}$ of dimension ${m}$.
\end{proof}

\begin{proof}[Proof of Proposition~\ref{thm:wp}]
By~\eqref{eq:stacked}, ${\B=\{w\in\W\mid Dw=d\}}$, that is, ${w\in\B}$ if and only if
\begin{equation}\label{eq:app_split}
D_{\mathrm{out}}w_{\mathrm{out}}=d-D_{\mathrm{in}}w_{\mathrm{in}} .
\end{equation}
(i) By Lemma~\ref{lem:affine}, the behavior of an open quadratic optimization program with input ${w_{\mathrm{in}}}$, decision variable ${w_{\mathrm{out}}}$, and nonsingular \gls{KKT} matrix is the graph of an affine map defined on the whole input space.
Suppose first that ${\B}$ is such a graph. Then ${\B\neq\emptyset}$, and~\eqref{eq:app_split} has a unique solution ${w_{\mathrm{out}}}$ for every ${w_{\mathrm{in}}}$. Uniqueness gives ${\ker D_{\mathrm{out}}=\{0\}}$, and existence gives
\[
d-D_{\mathrm{in}}w_{\mathrm{in}}\in\im D_{\mathrm{out}}\quad\text{for every }w_{\mathrm{in}},
\]
in which case ${d\in\im D_{\mathrm{out}}}$. Then ${\im D_{\mathrm{in}}\subseteq\im D_{\mathrm{out}}}$.
Conversely, assume ${\B\neq\emptyset}$, ${\ker D_{\mathrm{out}}=\{0\}}$, and ${\im D_{\mathrm{in}}\subseteq\im D_{\mathrm{out}}}$. For any ${w\in\B}$,
\[
d=D_{\mathrm{in}}w_{\mathrm{in}}+D_{\mathrm{out}}w_{\mathrm{out}}\in\im D_{\mathrm{out}},
\]
so that~\eqref{eq:app_split} is consistent for every ${w_{\mathrm{in}}}$ and has the unique solution
\[
w_{\mathrm{out}}=(D_{\mathrm{out}}\tp D_{\mathrm{out}})^{-1}D_{\mathrm{out}}\tp(d-D_{\mathrm{in}}w_{\mathrm{in}}) .
\]
This is also the unique minimizer of the least-squares program
\[
\min_{w_{\mathrm{out}}}\ \tfrac12\|D_{\mathrm{out}}w_{\mathrm{out}}+D_{\mathrm{in}}w_{\mathrm{in}}-d\|^2 ,
\]
which, up to a term independent of ${w_{\mathrm{out}}}$, is the open quadratic optimization program~\eqref{eq:oqp} with ${P=D_{\mathrm{out}}\tp D_{\mathrm{out}}}$, ${N=D_{\mathrm{out}}\tp D_{\mathrm{in}}}$, ${q=-D_{\mathrm{out}}\tp d}$, input ${w_{\mathrm{in}}}$, and no constraints. Hence, ${\B}$ is the behavior of this program, whose \gls{KKT} matrix ${D_{\mathrm{out}}\tp D_{\mathrm{out}}}$ is nonsingular.

(ii) If ${Dw=d}$ is consistent, then ${\B=w_0+\ker D}$ for any ${w_0\in\B}$, which is a singleton if and only if ${\ker D=\{0\}}$, that is, if and only if ${D}$ has full column rank.
If ${Dw=d}$ is inconsistent, then ${\B=\emptyset}$.

(iii) By definition, ${\B=\{\,Ez+e\mid z=\Xi z+b_\star\,\}}$, and ${z\mapsto Ez+e}$ is a bijection from the solution set of ${(I-\Xi)z=b_\star}$ onto ${\B}$. A square linear system has exactly one solution if and only if the matrix defining the system is nonsingular. Hence, ${\B}$ is a singleton if and only if ${1\notin\spec(\Xi)}$, in which case ${w_\star=(I-\Xi)^{-1}b_\star}$ is the unique solution of the loop equation and ${w=Ew_\star+e}$, with ${w_\star=R_\star w}$.
\end{proof}

\begin{proof}[Proof of Proposition~\ref{prop:gs}]
(i) Since ${K_1}$ is nonsingular, block elimination gives
\[
\begin{bmatrix}I&0\\-L\tp K_1^{-1}&I\end{bmatrix}
\begin{bmatrix}K_1&L\\L\tp&K_2\end{bmatrix}
=\begin{bmatrix}K_1&L\\0&K_2-L\tp K_1^{-1}L\end{bmatrix},
\]
and ${K_2-L\tp K_1^{-1}L=K_2(I-\Gamma)}$ by~\eqref{eq:gamma}. The left factor is unit lower triangular, so the matrix in~\eqref{eq:block} is nonsingular if and only if its Schur complement ${K_2(I-\Gamma)}$ is~\cite[Ch.~1]{zhang2005schur}, that is, ${K_2}$ being nonsingular, if and only if ${1\notin\spec(\Gamma)}$.

(ii) By~\eqref{eq:gamma}, ${\Gamma=XY}$ and ${\widetilde\Gamma=YX}$ with ${X:=K_2^{-1}L\tp}$ and ${Y:=K_1^{-1}L}$. If ${XYv=\theta v}$ with ${v\neq0}$ and ${\theta\neq0}$, then ${Yv\neq0}$ and ${YX(Yv)=\theta\,Yv}$, and likewise with ${X}$ and ${Y}$ exchanged. Thus, ${XY}$ and ${YX}$ have the same nonzero eigenvalues, and ${\rho(\Gamma)=\rho(\widetilde\Gamma)}$.

(iii) By (i), the system~\eqref{eq:block} has a unique solution ${w^\star}$, and its two block rows read
\[
w_2^\star=K_2^{-1}\big(r_2-L\tp w_1^\star\big),
\qquad
w_1^\star=K_1^{-1}\big(r_1-Lw_2^\star\big).
\]
Subtracting these identities from~\eqref{eq:iter} gives
\[
\begin{aligned}
w_2^{(j+1)}-w_2^\star&=-K_2^{-1}L\tp\big(w_1^{(j)}-w_1^\star\big),\\
w_1^{(j+1)}-w_1^\star&=-K_1^{-1}L\big(w_2^{(j+1)}-w_2^\star\big),
\end{aligned}
\]
and substituting the first identity into the second gives ${w_1^{(j+1)}-w_1^\star=\widetilde\Gamma\,(w_1^{(j)}-w_1^\star)}$.
Let ${\varepsilon_j:=w_1^{(j)}-w_1^\star}$, so that ${\varepsilon_j=\widetilde\Gamma^{\,j}\varepsilon_0}$. Then ${\varepsilon_j\to0}$ for every ${\varepsilon_0}$ if and only if ${\widetilde\Gamma^{\,j}\to0}$, that is, if and only if ${\rho(\widetilde\Gamma)<1}$~\cite{Varga2000}, and ${\rho(\widetilde\Gamma)=\rho(\Gamma)}$ by (ii). The convergence of ${w_2^{(j)}}$ then follows from the first recursion. Finally, by the definition of the induced norm,
\[
\|\varepsilon_j\|\le\|\widetilde\Gamma^{\,j}\|\,\|\varepsilon_0\| ,
\]
and, for every induced matrix norm~\cite{Varga2000},
\[
\lim_{j\to\infty}\|\widetilde\Gamma^{\,j}\|^{1/j}=\rho(\widetilde\Gamma)=\rho(\Gamma) .
\]
\end{proof}

\begin{proof}[Proof of Lemma~\ref{lem:error}]
Since ${K_2}$ is nonsingular, ${A_2}$ has full row rank, because ${(0,v)\in\ker K_2}$ for every ${v\in\ker A_2\tp}$, and ${P_2}$ is positive definite on ${\ker A_2}$~\cite[Sec.~10.1.1]{BoydVandenberghe2004}, that is, ${\widehat P_2\succ0}$.
By~\eqref{eq:iter} with ${w_1^{(0)}=0}$ and by the second block row of~\eqref{eq:block}, ${w_2^{(1)}=K_2^{-1}r_2}$ and ${w_2^\star=K_2^{-1}(r_2-L\tp w_1^\star)}$, so that, by~\eqref{eq:CZ},
\begin{equation}\label{eq:app_w2}
w_2^{(1)}-w_2^\star=K_2^{-1}L\tp w_1^\star=K_2^{-1}\begin{bmatrix}C\tp w_1^\star\\0\end{bmatrix}.
\end{equation}
We compute the primal component of the right-hand side by the null-space method~\cite[Ch.~16]{NocedalWright2006}. For ${\zeta\in\R^{n_2}}$, let ${(\delta\xi,\delta\lambda)}$ be the solution of ${K_2(\delta\xi,\delta\lambda)=(\zeta,0)}$, that is,
\[
P_2\delta\xi+A_2\tp\delta\lambda=\zeta,
\qquad
A_2\delta\xi=0 .
\]
The second equation gives ${\delta\xi=Z\eta}$ for some ${\eta}$, and multiplying the first by ${Z\tp}$ gives, since ${Z\tp A_2\tp=(A_2Z)\tp=0}$,
\[
\widehat P_2\eta=Z\tp\zeta,
\  \implies \ 
\delta\xi=Z\widehat P_2^{-1}Z\tp\zeta .
\]
Then~\eqref{eq:app_w2} with ${\zeta=C\tp w_1^\star}$ gives
\[
\xi^{(1)}-\xi^\star=Z\widehat P_2^{-1}Z\tp C\tp w_1^\star .
\]
Finally, the first block row of~\eqref{eq:block} reads
\[
K_1w_1^\star=r_1-C\xi^\star ,
\]
which is the \gls{KKT} system~\eqref{eq:kkt} of the first program at the input ${\xi^\star}$. By Theorem~\ref{thm:pd},
\[
C\tp w_1^\star=N_1\tp x^\star+M_1\tp\lambda_1^\star=\nabla V_1(\xi^\star) ,
\]
which proves~\eqref{eq:error}.
To establish the last claim, let ${S_1:=K_1-LK_2^{-1}L\tp}$. Eliminating ${w_2}$ from~\eqref{eq:block} gives
\[
S_1w_1^\star=r_1-LK_2^{-1}r_2 ,
\]
and ${S_1}$ is nonsingular, since the determinant of the matrix in~\eqref{eq:block} equals ${\det K_2\det S_1}$~\cite[Ch.~1]{zhang2005schur} and this matrix is nonsingular by Proposition~\ref{prop:gs}. Hence, ${w_1^\star=S_1^{-1}(r_1-LK_2^{-1}r_2)}$ ranges over ${\R^{n_1+p_1}}$ as ${r_1=(-q_1,b_1)}$ varies. Since ${Z\widehat P_2^{-1}}$ has full column rank, ${\xi^{(1)}-\xi^\star=Z\widehat P_2^{-1}(CZ)\tp w_1^\star}$ vanishes for all data if and only if ${CZ=0}$, that is, ${\ker A_2\subseteq\ker N_1\cap\ker M_1}$.
\end{proof}

\begin{proof}[Proof of Theorem~\ref{thm:pd}]
(i) Let ${w:=(x,\lambda)}$ and
\[
d(u):=\begin{bmatrix}q+Nu\\Mu-b\end{bmatrix}=d(0)+Cu ,
\]
so that~\eqref{eq:kkt} reads ${Kw=-d(u)}$, and the Lagrangian of~\eqref{eq:oqp} reads
\[
\begin{aligned}
\mathcal L(x,\lambda,u)&=f(x,u)+\lambda\tp(Ax+Mu-b)\\
&=\tfrac12w\tp Kw+d(u)\tp w .
\end{aligned}
\]
By Lemma~\ref{lem:affine}, ${w^\star:=(x^\star(u),\lambda^\star(u))=-K^{-1}d(u)}$. Since ${Ax^\star(u)+Mu=b}$, the value function is
\[
\begin{aligned}
V(u)&=\mathcal L(w^\star,u)\\
&=\tfrac12w^{\star\top}Kw^\star+d(u)\tp w^\star\\
&=-\tfrac12d(u)\tp K^{-1}d(u),
\end{aligned}
\]
a quadratic function of ${u}$. Since ${K^{-1}}$ is symmetric and ${\partial_ud=C}$, differentiating gives
\[
\begin{aligned}
\nabla V(u)&=-C\tp K^{-1}d(u)\\ &=C\tp w^\star\\
&=N\tp x^\star(u)+M\tp\lambda^\star(u)\\ &=y,\\
\nabla^2V&=-C\tp K^{-1}C .
\end{aligned}
\]
Hence, ${\B^{\mathrm{pd}}}$ is the graph of ${\nabla V}$, and ${y=Su+s}$ with ${S=\nabla^2V}$ and ${s:=-C\tp K^{-1}d(0)}$.

(ii) Let $${\mathcal G:=\{(u,Su)\mid u\in\R^m\}}$$ be the direction space of ${\B^{\mathrm{pd}}}$. Then ${\mathcal G}$ has dimension ${m}$ and, since ${S}$ is symmetric,
\[
\begin{aligned}
\omega\big((u_1,Su_1),(u_2,Su_2)\big)&=(Su_2)\tp u_1-(Su_1)\tp u_2\\
&=u_1\tp(S-S\tp)u_2=0 .
\end{aligned}
\]
Thus, ${\mathcal G}$ is Lagrangian, as is the graph of every symmetric linear map~\cite{arnold1989}.
(iii) If ${N=0}$, the function equal to ${\tfrac12x\tp Px+q\tp x}$ on the affine set ${\{(x,u)\mid Ax+Mu=b\}}$ and to ${+\infty}$ elsewhere is jointly convex in ${(x,u)}$, and ${V}$ is its partial minimization over ${x}$, which preserves convexity~\cite[Sec.~3.2.5]{BoydVandenberghe2004}. Hence, ${V}$ is convex and ${S\succeq0}$.
If ${M=0}$, the feasible set ${\mathcal X}$ does not depend on ${u}$, so that ${V}$ is the pointwise infimum over ${x\in\mathcal X}$ of the functions
\[
u\mapsto\tfrac12x\tp Px+q\tp x+u\tp N\tp x ,
\]
which are affine in ${u}$. Hence, ${V}$ is concave~\cite[Sec.~3.2.3]{BoydVandenberghe2004} and ${S\preceq0}$.
To establish the last claim, consider ${P=I_2}$, ${A=M=\begin{bmatrix}0&1\end{bmatrix}}$, ${N=\diag(1,0)}$, ${q=0}$, and ${b=0}$. Then ${x^\star(u)=-u}$ and ${V(u)=-\tfrac12u_1^2+\tfrac12u_2^2}$, so that ${S=\diag(-1,1)}$ is indefinite.
\end{proof}

\begin{proof}[Proof of Theorem~\ref{thm:port}]
Let ${f_i(x_i,u_i,c_i)}$ and ${\mathcal X_i(u_i,c_i)}$ denote the objective and the feasible set of the ${i}$-th program, ${i=1,2}$, so that ${V_i(u_i,c_i)=\min\{f_i(x_i,u_i,c_i)\mid x_i\in\mathcal X_i(u_i,c_i)\}}$, the minimum being attained by (H1) and Lemma~\ref{lem:affine}.

(i) Minimizing first over ${(x_1,x_2)}$ and then over ${(c_1,c_2)}$ gives
\[
\begin{aligned}
V_{12}(u)&=\inf_{c_1+c_2=c_{\mathrm{s}}}\ \inf_{\substack{x_1\in\mathcal X_1(u_1,c_1)\\ x_2\in\mathcal X_2(u_2,c_2)}}\ \sum_{i=1,2}f_i(x_i,u_i,c_i)\\
&=\inf_{c_1+c_2=c_{\mathrm{s}}}\big\{V_1(u_1,c_1)+V_2(u_2,c_2)\big\}\\
&=(V_1\,\Box\,V_2)(u).
\end{aligned}
\]
By (H1) and Theorem~\ref{thm:pd}, ${V_1}$ and ${V_2}$ are quadratic, so that
\[
g(c_1):=V_1(u_1,c_1)+V_2(u_2,c_{\mathrm{s}}-c_1)
\]
\textit{i.e.}, is a quadratic function of ${c_1}$ with Hessian ${H_{cc}}$, which is positive definite by (H3). Hence, ${g}$ attains its infimum at a unique point, characterized by
\[
\nabla g(c_1)=\nabla_{c_1}V_1(u_1,c_1)-\nabla_{c_2}V_2(u_2,c_{\mathrm{s}}-c_1)=0 ,
\]
which together with ${c_2=c_{\mathrm{s}}-c_1}$ is~\eqref{eq:portgen}.

(ii) For every ${(x_1,x_2,c_1,c_2)}$ feasible for the interconnected program,
\[
\begin{aligned}
f_1(x_1,u_1,c_1)+f_2(x_2,u_2,c_2)&\ge V_1(u_1,c_1)+V_2(u_2,c_2)\\
&\ge V_{12}(u),
\end{aligned}
\]
where the first inequality is an equality if and only if ${x_i}$ is optimal for the ${i}$-th program at ${c_i}$, ${i=1,2}$, and the second if and only if ${(c_1,c_2)}$ satisfies~\eqref{eq:portgen}, by (i). Since ${V_{12}(u)}$ is finite, ${(x_1,x_2,c_1,c_2)}$ is optimal if and only if both inequalities are indeed equalities.

(iii) By (H1) and Theorem~\ref{thm:pd}, the function
\[
(u,c_1)\mapsto V_1(u_1,c_1)+V_2(u_2,c_{\mathrm{s}}-c_1)
\]
is quadratic with Hessian ${H}$, that is, with ${c:=c_1}$, it reads
\[
\tfrac12u\tp H_{uu}u+u\tp H_{uc}c+\tfrac12c\tp H_{cc}c+h_u\tp u+h_c\tp c+h_0
\]
for suitable ${h_u}$, ${h_c}$, and ${h_0}$, and by (i) ${V_{12}(u)}$ is its minimum over ${c}$. By (H3), the minimizer is
\[
c^\star(u)=-H_{cc}^{-1}(H_{cu}u+h_c),
\]
and substituting it gives a quadratic function of ${u}$ with Hessian ${H_{uu}-H_{uc}H_{cc}^{-1}H_{cu}}$, the Schur complement of ${H_{cc}}$ in ${H}$~\cite[App.~A.5.5]{BoydVandenberghe2004}.
\end{proof}

\begin{proof}[Proof of Proposition~\ref{prop:lift}]
As in the proof of Lemma~\ref{lem:error}, nonsingularity of ${K}$ implies that ${A}$ has full row rank, so that ${\mathcal X(u)\neq\emptyset}$, and that ${\widehat P\succ0}$. Let ${x^\star:=x^\star(u,b)}$ and ${\lambda^\star:=\lambda^\star(u,b)}$.
If ${d=0}$, then ${\mathcal X(u)=\{x^\star\}}$, the integral in~\eqref{eq:freeenergy} is the value ${e^{-V(u,b)/T}}$ of the integrand at ${x^\star}$, and ${p_T}$ is the unit mass at ${x^\star}$, so that (i) and (ii) hold with the convention ${\det\widehat P=1}$ for the void matrix ${\widehat P}$.

Next, assume ${d\geq1}$. Since ${\mathcal X(u)=x^\star+\ker A}$, every ${x\in\mathcal X(u)}$ is ${x=x^\star+Z\eta}$ for a unique ${\eta\in\R^d}$. The first block row of~\eqref{eq:kkt} gives ${\nabla_xf(x^\star,u)=-A\tp\lambda^\star}$, so that
\[
Z\tp\nabla_xf(x^\star,u)=-(AZ)\tp\lambda^\star=0
\]
and, ${f}$ being quadratic in ${x}$ with Hessian ${P}$,
\[
f(x^\star+Z\eta,u)=V(u,b)+\tfrac12\eta\tp\widehat P\eta .
\]
By the definition of the integral over ${\mathcal X(u)}$ with ${x_0=x^\star}$ and the normalization of the Gaussian density,
\begin{equation}\label{eq:display}
\begin{aligned}
\int_{\mathcal X(u)}e^{-f(x,u)/T}\,dx
&=e^{-V(u,b)/T}\int_{\R^d}e^{-\eta\tp\widehat P\eta/(2T)}\,d\eta\\
&=e^{-V(u,b)/T}(2\pi T)^{d/2}(\det\widehat P)^{-1/2}.
\end{aligned}
\end{equation}
The integral is therefore finite, and ${p_T}$ is the density of ${x^\star+Z\eta}$ with ${\eta}$ Gaussian with mean zero and covariance ${T\widehat P^{-1}}$. Hence, ${p_T}$ is Gaussian with mean ${x^\star}$ and covariance ${TZ\widehat P^{-1}Z\tp}$, which proves (i), and taking ${-T\log}$ of~\eqref{eq:display} proves (ii).
Since ${A}$ and ${P}$ do not depend on ${(u,b)}$, neither does the last term in (ii), so that ${\mathcal F_T}$ and ${V}$ have the same gradients. Theorem~\ref{thm:pd} gives ${\nabla_uV=y}$. Differentiating ${Ax^\star+Mu=b}$ with respect to ${b}$ gives ${A\,\partial_bx^\star=I}$, so that
\[
\nabla_bV=(\partial_bx^\star)\tp\nabla_xf(x^\star,u)=-(\partial_bx^\star)\tp A\tp\lambda^\star=-\lambda^\star ,
\]
which proves (iii).
\end{proof}

\begin{proof}[Proof of Theorem~\ref{thm:value}]
(i) Adding ${\kappa_i}$ to the objective of the ${i}$-th program adds ${\kappa_i}$ to ${V_i}$, which leaves ${\nabla V_i}$, hence ${\B_i^{\mathrm{pd}}}$, unchanged. For every set ${\mathcal D}$ and functions ${f_1,f_2}$ on ${\mathcal D}$,
\[
\inf_{\mathcal D}\{(f_1+\kappa_1)+(f_2+\kappa_2)\}=\inf_{\mathcal D}\{f_1+f_2\}+\kappa_1+\kappa_2 .
\]

(ii) For each ${u}$, the value behaviors ${\B_1^{\mathrm{v}}}$ and ${\B_2^{\mathrm{v}}}$ contain exactly one point with input ${u}$, namely ${(u,\nabla V_i(u),V_i(u))}$, ${i=1,2}$, and~\eqref{eq:valadd} maps this pair to
\[
\begin{aligned}
&(u,\nabla V_1(u)+\nabla V_2(u),V_1(u)+V_2(u))\\
&\qquad=(u,\nabla(V_1+V_2)(u),(V_1+V_2)(u)),
\end{aligned}
\]
the unique point of the value behavior of ${V_1+V_2}$ with input ${u}$.

(iii) Since ${\widehat P_{12}\succ0}$, its leading principal submatrix ${\widehat P_1}$ and the Schur complement ${\widehat P_2-\widehat L\tp\widehat P_1^{-1}\widehat L}$ of ${\widehat P_1}$ in ${\widehat P_{12}}$ are positive definite, and
\[
\det\widehat P_{12}=\det\widehat P_1\,\det\big(\widehat P_2-\widehat L\tp\widehat P_1^{-1}\widehat L\big)
\]
by the Schur determinant formula~\cite[Ch.~1]{zhang2005schur}. Taking logarithms proves the identity.
(iv) By the integrability assumption, all expectations below are finite, and linearity of the expectation gives
\[
\begin{aligned}
\mathcal{F}_1+\mathcal{F}_2&=2\E_\mu[\log\mu]-\E_\mu[\log\psi_1]-\E_\mu[\log\psi_2]\\
&=\mathcal{F}_{12}+\E_\mu[\log\mu]=\mathcal{F}_{12}-\Ent(\mu).
\end{aligned}
\]
\end{proof}

\end{document}